\documentclass{amsart}
\usepackage{amsmath,amssymb,amsthm,verbatim}
\newtheorem{theorem}{Theorem}
\newtheorem{corollary}{Corollary}
\newtheorem{proposition}{Proposition}
\newtheorem{lemma}{Lemma}
\begin{document}
\title[Generators for Decompositions of Tensor Products of Modules]{Generators for Decompositions of Tensor Products of Modules associated with standard Jordan partitions}
\author{Michael~J.~J.~Barry}
\address{Department of Mathematics\\
Allegheny College\\
Meadville, PA 16335}
\email{mbarry@allegheny.edu}
\thanks{}


\begin{abstract}
If $K$ is a field of finite characteristic $p$, $G$ is a cyclic group of order $q=p^\alpha$, $U$ and $W$ are indecomposable $KG$-modules with $\dim U=m$ and $\dim W=n$, and $\lambda(m,n,p)$ is a standard Jordan partition of $ m n$, we describe how to find a generator for each of the indecomposable components of the $KG$-module $U \otimes W$.
\end{abstract}
\maketitle
\section{Introduction}
Let $p$ be a prime number, $K$ a field of characteristic $p$, and $G$ a cyclic group of order $q=p^\alpha$, where $\alpha$ is a positive integer.  It is well-known that there are exactly $q$ isomorphism classes of indecomposable $KG$-modules and that such modules are cyclic and uniserial~\cite[p. 24--25]{A1986}.  Let $\{V_1,\dots,V_q\}$ be a set of representatives of these isomorphism classes with $\dim V_i=i$.  Many authors have investigated the decomposition of the $K G$-module $V_m \otimes V_n$, where $m \leq n$, into a direct sum of indecomposable $K G$-modules---for example, in order of publication, see \cite{G1962}, \cite{S1964}, \cite{L1974},  \cite{N1995}, \cite{R1979}, \cite{H2003},  \cite{N2008}, and \cite{B2011_0}. From the works of these authors, it is well-known that $V_m \otimes V_n$ decomposes into a direct sum $V_{\lambda_1} \oplus V_{\lambda_2} \dots \oplus V_{\lambda_m}$ of $m$ indecomposable $KG$-modules where $\lambda_1 \geq \lambda_2 \geq \dots \geq \lambda_m>0$, but that the dimensions $\lambda_i$ of the components depend on the characteristic $p$.  Now $\lambda(m,n,p)=(\lambda_1, \lambda_2,\dots,\lambda_m)$ is called a \textbf{Jordan partition} of $m n$, and $\lambda(m,n,p)$ is said to be \textbf{standard} exactly when $\lambda_i=m+n+1-2 i$ for every integer  $i \in [1,m]$.  So when $\lambda(m,n,p)$ is standard,
\[V_m \otimes V_n \cong \bigoplus_{i=1}^m V_{n+m+1-2 i}.\]
Necessary and sufficient conditions on $m$, $n$, and $p$ for $\lambda(m,n,p)$ to be standard were given in~\cite{B2015}.

Fix a generator $g$ of $G$.    There is a basis $u_1$, $u_2$, \dots, $u_m$ of $V_m$ on which the action of $g$ is given by $g u_1=u_1$ and $g u_i=u_{i-1}+u_i$ when $i>1$.  Note that $(g-1)^i u_m=u_{m-i}$, and so $u_m$ generates $V_m$ as a $K G$-module.  Similarly there is a  basis $w_1$, $w_2$, \dots, $w_n$ of $V_n$, with $V_n$ generated as a $K G$-module by $w_n$, on which the action of $g$ is given by $g w_1=w_1$ and $g w_i=w_{i-1}+w_i$ when $i>1$.  Clearly $\{v_{i,j}=u_i \otimes w_j \mid 1 \leq i \leq m, 1 \leq j \leq n\}$ is a basis of $V_m \otimes V_n$ over $K$.  Shortly we will see that $\mathcal{B}=\{f_{i,j}=u_i \otimes g^{n-i} w_j\mid 1 \leq i \leq m, 1 \leq j \leq n\}$ is also a basis of $V_m \otimes V_n$ that turns out to be easier to work with.  We will specify, in terms of $\mathcal{B}$, $m$ elements $y_1$, $y_2$, \dots, $y_m$ in $V_m \otimes V_n$ such that, when $\lambda(m,n,p)$ is standard, $K G y_i \cong V_{n+m+1-2 i}$ and $V_m \otimes V_n$ is an internal direct sum of the indecomposable modules $K G y_i$. ($m \geq 2$)

Barry did this for the special standard partition $\lambda(m,n,p)$ with $m+n \leq p+1$ in ~\cite{B2011}, and  Glasby, Praeger, and Xia did it for a subset of standard partitions that properly includes the case that Barry dealt with in~\cite{GPX}.

We now describe the organization of the paper.  In Section~\ref{Prelim}, we show how to calculate in $V_m \otimes V_n$, state our results in Section~\ref{Results}, work out an example in Section~\ref{Example}, dealt with characteristic $2$ in Section~\ref{Char2}, and then the rest of the paper deals with the odd characteristic case.

\section{Preliminaries}\label{Prelim}
\begin{lemma}\label{f_{i,j}}
The set $\mathcal{B}=\{f_{i,j}=u_i \otimes g^{n-i} w_j \mid 1 \leq i\leq m, 1 \leq j \leq n\}$ is an $K$-basis for $V_m \otimes V_n$, and $(g-1) f_{i,j}=f_{i-1,j}+f_{i,j-1}$, where we understand that $f_{k,\ell}=0$ if $k<1$ or $\ell<1$.
\end{lemma}

\begin{proof}
Since $f_{i,j}=v_{i,j} +\sum_{k+\ell <i+j} \alpha_{k,\ell}v_{k,\ell}$, the linear independence of $\mathcal{B}$ follows from the linear independence of $\{v_{i,j} \mid 1 \leq i \leq m, 1 \leq j \leq n\}$. Also
\begin{align*}
g f_{i,j}&=g u_i \otimes g^{n-i+1} w_j\\
&=(u_i+u_{i-1}) \otimes g^{n-i}(w_j+w_{j-1})\\
&=u_i \otimes g^{n-i} w_j + u_i \otimes g^{n-i} w_{j-1} +u_{i-1} \otimes g^{n-i}(w_j+w_{j-1})\\
&=f_{i,j}+f_{i,j-1}+u_{i-1} \otimes g^{n-(i-1)} w_j\\
&=f_{i,j}+f_{i,j-1}+f_{i-1,j}
\end{align*}
Thus
$(g-1) f_{i,j}=f_{i,j-1}+f_{i-1,j}$.
\end{proof}

For an integer $k \in [1,m+n-1]$, define
\[ F_k=\langle f_{i,j} \mid i+j \leq k+1 \rangle, \text{ and } D_k=\langle f_{i,j} \mid i+j=k+1 \rangle.\]
Also for convenience define $D_k=\{0\}$ for $k<1$.
Note that $F_k=\langle v_{i,j} \mid i+k \leq k+1 \rangle$ and
$D_k \subset F_k$.

\begin{lemma}
For each integer $i \in [1,m]$, define $x_i=\sum_{j=1}^i (-1)^{j-1} f_{j,i+1-j} \in D_i$.  Then $\{x_1,x_2,\dots,x_m\}$ is linearly independent and $(g-1) x_i=0$ for all $i$.
\end{lemma}
\begin{proof}
First $\{x_1,x_2,\dots,x_m\}$ is linearly independent because $x_i \in D_i$ and $F_m=D_1 \oplus \dots \oplus D_m$.  Also
\begin{align*}
(g-1) x_i&=\sum_{j=1}^i (-1)^{j-1}f_{j,i-j} +\sum_{j=1}^i (-1)^{j-1}f_{j-1,i+1-j}\\
&=\sum_{j=1}^i (-1)^{j-1}f_{j,i-j}+\sum_{j=0}^{i-1}(-1)^j f_{j,i-j}\\
&=f_{i,0}+f_{0,i}\\
&=0.
\end{align*}
\end{proof}

Then, by Lemma~\ref{f_{i,j}},  $(g-1)^r(D_k) \subseteq D_{k-r}$ and  
\[(g-1)^r(f_{i,j})=\sum_{k=0}^r \binom{r}{k} f_{i+k-r,j-k},\] 
where we understand that $f_{i+k-r,j-k}=0$ if $i+k-r<0$ or $j-k<0$.

Hence $(g-1)^{m+n-2 k}(D_{m+n-k}) \subseteq D_k$ when $1 \leq k \leq m$
and
\[(g-1)^{m+n-2 k}(f_{i,j})=\sum_{\ell=0}^{m+n-2 k} \binom{m+n-2 k}{\ell}f_{i+\ell-m-n+2 k,j-\ell}.\]

Assume that $m \leq n$.  Denote the ordered $F$-basis
\[(f_{m-k+1,n}, f_{m-k+2,n-1} \dots, f_{m,n-k+1})\] of $D_{m+n-k}$ by $\mathcal{B}_{m+n-k}$
and the ordered $F$-basis $(f_{1,k},f_{2,k-1},\dots,f_{k,1})$ of $D_k$ by
$\mathcal{B}_{k}$.  In the case where $m=n$ and $k=m$,  $\mathcal{B}_{m+n-k}=\mathcal{B}_{k}$.

\begin{lemma}
Let $A_k(m,n)$ be the matrix with respect to the ordered $F$-bases $\mathcal{B}_{m+n-k}$ and $\mathcal{B}_{k}$ of $D_{m+n-k}$ and $D_k$, respectively.  Then
\[A_k(m,n)=\begin{pmatrix}
\binom{m+n-2 k}{n-k} & \binom{m+n-2 k}{n-k-1} & \dots & \binom{m+n-2 k}{n+1-2k}\\
\binom{m+n-2 k}{n+1-k} & \binom{m+n-2 k}{n-k} & \dots & \binom{m+n-2 k}{m+2-2k}\\
\vdots & \vdots & \ddots & \vdots \\
\binom{m+n-2 k}{n-1} & \binom{m+n-2 k}{n-2} & \dots & \binom{m+n-2 k}{n-k}
\end{pmatrix}.\]
\end{lemma}

\begin{proof}
A typical element in $\mathcal{B}_{m+n-k}$ is $f_{m-k+t,n-t+1}$, $1 \leq t \leq k$, while a typical element in $\mathcal{B}_{k}$ is $f_{s,k+1-s}$, $1 \leq s \leq k$.

Now
\[(g-1)^{m+n-2k}(f_{m-k+t,n-t+1})=
\sum_{\ell=0}^{m+n-2k} \binom{m+n-2k}{\ell} 
f_{\ell+k+t-n,n-t+1-\ell}\]

When $\ell+k+t-n=s$ (and $n-t-\ell+1=k+1-s$), $\ell=n+s-k-t$.  Thus the coefficient $f_{s,k+1-s}$ in the expansion of $(g-1)^{m+n-2k}(f_{m-k+t,n-t+1})$ is 
\[\binom{m+n-2k}{n+s-k-t}.\]

This proves our lemma.
\end{proof}

\section{Statement of Results}\label{Results}

For an integer $k \in [1,m]$, define the $k \times 1$ column vector $C_k$ to be the coordinate matrix of of $x_k$ with respect to the basis $\mathcal{B}_k$ of $D_k$.  Then $C_k$ consists of alternating $1$'s and $-1$'s.  Then define the $k \times 1$ column vector $B_k$ by $B_k=\text{adj}(A_k) C_k$, where $\text{adj}(A_k(m,n))$ is the classical adjoint of $A_k(m,n)$ (so $A_k(m,n) \text{adj}(A_k(m,n))=(\det A_k(m,n))I_k=\text{adj}(A_k(m,n))A_k(m,n)$).

\begin{theorem}\label{Th1}
With $A_k(m,n)$ and $B_k$ defined as above, and $y_k$ defined by
\[y_k=\sum_{i=1}^k b_{i 1}f_{n-k+i,m+1-i}, \]
the equation $(g-1)^{n+m-2k} \cdot y_k=(\det A_k(m,n)) x_k$ holds.
\end{theorem}

\begin{proof}
Since by~\cite[p. 392]{Anton2000},
\[[(g-1)^{n+m-2k} \cdot y_k]_{\mathcal{B}_k}=
[(g-1)^{n+m-2k}]_{\mathcal{B}_k,\mathcal{B}_{m+n-k}}[y_k]_{\mathcal{B}_{m+n-k}},
\]
we have
\[[(g-1)^{n+m-2k} \cdot y_k]_{\mathcal{B}_k}=A_k(m,n) B_k=A_k(m,n) \text{adj}(A_k(m,n)) C_k=(\det A_k(m,n)) C_k.\]
But $[x_k]_{\mathcal{B}_k}=C_k$, which implies that $(g-1)^{n+m-2k} \cdot y_k=(\det A_k(m,n)) x_k$.
\end{proof}

\begin{corollary}\label{Cor}
When $\lambda(m,n,p)$ is standard, then $A_k(m,n)$ is invertible for every integer $k \in [1,m]$, and if $y_1$, $y_2$, \dots, $y_m$ are defined as in Theorem~\ref{Th1}, $K G y_k \cong V_{n+m+1-2 k}$ ($1 \leq k \leq m$) and  
\[V_m \otimes V_n =K G y_1 \oplus K G y_2 \oplus \dots \oplus K G y_m.\]
\end{corollary}

Once we prove that $A_k(m,n)$ is invertible for every integer $k \in [1,m]$, the rest of the proof follows the proof in~\cite{B2011} or~\cite[Theorem 2]{GPX}.

\section{Example}\label{Example}
We illustrate Theorem~\ref{Th1} when $m=4$, $n=5$, and $k=3$.  In this case $x_3=f_{1,3}-f_{2,2}+f_{3,1}$,
\[A_3(4,5)=\begin{pmatrix} \binom{3}{2} & \binom{3}{1} & \binom{3}{0}\\
\binom{3}{3} & \binom{3}{2} & \binom{3}{1}\\
\binom{3}{4} & \binom{3}{3} & \binom{3}{2} \end{pmatrix}=
\begin{pmatrix} 3 & 3 & 1 \\ 1 & 3 & 3 \\ 0 & 1 & 3 \end{pmatrix},\]
$\det A_3(4,5)=10$,  and
\[\text{adj}(A_3(4,5))C_3=
\begin{pmatrix}
6 & -8 & 6 \\ -3 & 9 &-8 \\ 1 & -3 & 6
\end{pmatrix}
\begin{pmatrix}
1 \\ -1 \\ 1
\end{pmatrix}=
\begin{pmatrix}
20 \\ -20 \\ 10
\end{pmatrix}.\]
Thus $y_3=20 f_{2,5}-20 f_{3,4}+10 f_{4,3}$ and
\begin{align*}
(g-1)^3 &\cdot (20 f_{2,5}-20 f_{3,4}+10 f_{4,3})\\
&=20 \sum_{k=0}^3 \binom{3}{k} f_{k-1,5-k} -20 \sum_{k=0}^3 \binom{3}{k} f_{k,4-k}
+10 \sum_{k=0}^3 \binom{3}{k} f_{k+1,3-k}\\
&=20(3 f_{1,3}+f_{2,2})-20 (3 f_{1,3} +3 f_{2,2}+f_{3,1})=10 (f_{1,3}+3 f_{2,2}+3 f_{3,1})\\
&=10 f_{1,3}-10 f_{2,2}+10 f_{3,1}\\
&=(\det A_3(4,5)) x_3.
\end{align*}

\section{Proof in characteristic $2$}\label{Char2}

By~\cite{B2015}, $\lambda(m,n,2)$ is standard with $1<m \leq n$ if{f} either $(m,n)=(2,n)$ with $n \geq 3$ odd or $(m,n)=(3,6+4 r)$ where $r$ is a non-negative integer.

When $n \geq 3$ is odd, $A_1(2,n)=(\binom{2+n-2}{n-1}=(n)$ and
\[A_2(2,n)=\begin{pmatrix}
\binom{n-2}{n-2} & \binom{n-2}{n-3}\\
\binom{n-2}{n-1} & \binom{n-2}{n-2}
\end{pmatrix}=
\begin{pmatrix}
1 & n-2 \\ 0 & 1
\end{pmatrix}.\]
Hence both are invertible in $K$.

When $n=6 +4 r$ where $r$ is a non-negative integer, $A_1(3,n)=(\binom{3+n-2}{n-1})=(\binom{n+1}{n-1})$,
\[A_2(3,n)=\begin{pmatrix}
\binom{n-1}{n-2} & \binom{n-1}{n-3}\\
\binom{n-1}{n-1} & \binom{n-1}{n-2}
\end{pmatrix}=
\begin{pmatrix}
n-1 & \binom{n-1}{n-3} \\ 1 & n-1
\end{pmatrix},\]
and
\[A_3(3,n)=
\begin{pmatrix}
\binom{n-3}{n-3} & \binom{n-3}{n-4} & \binom{n-3}{n-5} \\
\binom{n-3}{n-2} & \binom{n-3}{n-3} & \binom{n-3}{n-4}\\
\binom{n-3}{n-1} & \binom{n-3}{n-2} & \binom{n-3}{n-3}
\end{pmatrix}=
\begin{pmatrix}
1 & n-3 & \binom{n-3}{n-5}\\ 0 & 1 & n-3 \\ 0 & 0 & 1
\end{pmatrix}.\]
Since $\binom{n+1}{n-1}=1$ and $\binom{n-1}{n-3}=0$ in $K$, these matrices are invertible.

\section{Standard Jordan partitions in odd characteristic}\label{SJP}
For the remainder of this paper, $p$ is a fixed odd prime.  Define
\[S'_0=\{(k,d) \in \mathbb{N} \times \mathbb{N}\mid 1<k \leq d \leq p+1-k\} \cup \{(k,p+k-1) \mid 1<k \leq (p+1)/2\},\]
and $S_0=\{(a,b +r p) \mid (a,b) \in S'_0, r \in \{0\} \cup \mathbb{N}\}$.

For an integer $t \geq 1$, define $S'_t=(T_1 \setminus T_2) \cup T_3$ where
\[T_1=\{(i p^{t}+(p^{t} \pm 1)/2,j p^{t}+(p^{t} \pm 1)/2) \mid i, j \in \mathbb{N}, 1 \leq i  \leq j \leq p-i-1\},\]
\[T_2=\{(i p^{t}+(p^{t} + 1)/2,i p^{t}+(p^{t} - 1)/2) \mid i \in \mathbb{N}, 1 \leq i \leq (p-1)/2\},\]
and
\[T_3=\{(i p^{t}+(p^{t} + 1)/2,i p^{t}+(p^{t} - 1)/2+p^{t+1}) \mid i \in \mathbb{N}, 1 \leq i \leq (p-1)/2\},\]
and $S_t=\{(a,b+r p^{t+1}) \mid (a,b) \in S'_t, r \in \{0\} \cup \mathbb{N}\}$.

Then by~\cite{B2015}, $S= \cup_{t \geq 0} S_t$ is the set of ordered pairs $(m,n)$ of positive integers with $1<m \leq n$ such that $\lambda(m,n,p)$ is standard.  Note that for each $(m,n) \in S$, neither $m$  nor $n$ is a power of $p$.

\section{Proof in odd characteristic}

To prove Corollary~\ref{Cor}, we must show that  $A_k(m,n)$, which is a matrix with entries in the prime field of $K$, is invertible for every integer $k \in [1,m]$ when $(m,n) \in S$.

We know, when $1 \leq k \leq m$, by~\cite{GPX} that
\[d_k(m,n)=\det A_k(m,n)=\prod_{\ell=0}^{k-1} \frac{\binom{m+n-2 k+\ell}{n-k}}{\binom{n-k+\ell}{n-k}}
=\prod_{\ell=0}^{k-1} \frac{(m+n-2 k+\ell)! \ell!}{(n-k+\ell)!(m-k+\ell)!}.\]

In particular, $d_m(m,n)=1$ because $A_m(m,n)$ is upper triangular with $1$'s along the diagonal.

Denote the exact power of $p$ dividing a non-zero integer $w$ by $\nu_p(w)$.  We will use extensively a  theorem of Kummer~\cite{Gran1997} which states that $\nu_p(\binom{n}{m})$ is the number of `carries' required to add $m$ and $n-m$ in base-$p$.

By~\cite[Lemma 12]{GPX2},
\[\binom{m+n-k-1}{k} d_{k+1}(m,n)=\binom{m+n-2 k-2}{n-k-1}d_k(m,n)\] 
when $0 \leq k \leq m-1$, where $d_0(m,n)$ is defined to be $1$.

We need to show that $\nu_p(d_k(m,n))=0$ for every integer $k \in [1,m]$.  Since $d_0(m,n)=1$, it suffices to prove the following result.

\begin{proposition}\label{Prop}
For $(m,n) \in S$, 
\[\nu_p \left(\binom{m+n-k-1}{k} \right)=\nu_p \left(\binom{m+n-2 k-2}{m-k-1} \right) \qquad (*)\] for every positive integer $k$ in the interval $[0,m-1]$. 
\end{proposition}

Note that since $d_m(m,n)=1$, we could restrict $k$ to $[0,m-2]$, but in the inductive step, having $k \in [0,m-1]$ proves useful.

We now outline the plan of the proof.  For a non-negative integer $t$, let $P(t)$ be the statement: Proposition~\ref{Prop} holds for every $(m,n) \in S_t$. First we prove the base case $t=0$, so $(m,n) \in S_0$, which means that $2 \leq m<p$, in Section~\ref{base}.  Then let $t>0$ and assume $P(t-1)$ is true.  In Section~\ref{inductstep} we use use our inductive hypothesis to show that Proposition~\ref{Prop} holds for $m=n=p^t+\frac{p^t+1}{2}$.  In subsequent sections  we use this result and our inductive hypothesis to show that Proposition~\ref{Prop} holds for all the other $(m,n) \in S_t$, which are listed below.
\begin{enumerate}
\item  $(m,n)=(ip^t+\frac{p^t+1}{2},ip^t+\frac{p^t+1}{2})$ where $1 < i \leq \frac{p-1}{2}$ \label{it1}
\item $(m,n)=(ip^t+\frac{p^t+1}{2},jp^t+\frac{p^t+1}{2})$ where $1 \leq i <j \leq p-i-1$
\item $(m,n)=(ip^t+\frac{p^t-1}{2},jp^t+\frac{p^t+1}{2})$ where $1 \leq i \leq j \leq p-i-1$
\item $(m,n)=(ip^t+\frac{p^t-1}{2},jp^t+\frac{p^t-1}{2})$ where $1 \leq i \leq j \leq p-i-1$
\item $(m,n)=(ip^t+\frac{p^t+1}{2},jp^t+\frac{p^t-1}{2})$ where $1 \leq i <j \leq p-i-1$
\item $(m,n)=(ip^t+\frac{p^t+1}{2},ip^t+\frac{p^t-1}{2}+p^{t+1})$ where $1 \leq  i \leq \frac{p-1}{2}$ \label{it2}
\item $(m,n+r p^{t+1})$, where $(m,n)$ comes from Cases~\ref{it1}--\ref{it2} and $r$ is a non-negative integer
\end{enumerate}
Then we will have shown that $P(t-1)$ implies $P(t)$ from which the result follows by appeal to mathematical induction.

\section{The case of $m<p$}\label{base}

Proposition~\ref{Prop} holds for $(m,n)\in S_0$.

\begin{proof}
Suppose that $1 < m \leq n \leq p+1-m$.  When $k=0$, $m-1 \leq n-1 <m+n-2\leq p-1$.  Hence
\[\nu_p \left(\binom{m+n-1}{0}\right)=0=\nu_p \left(\binom{m+n-2}{m-1} \right).\] 

Now assume $1 \leq k \leq m-1$.  Then
\[m+n-2 k-2<m+n-k-1 \leq p-k<p\] and
\[\nu_p \left(\binom{m+n-k-1}{k}\right)=0=\nu_p \left(\binom{m+n-2 k-2}{m-k-1} \right)\] for $k \in [1,m-1]$.

Suppose that $1<m \leq (p+1)/2$, $n=p+m-1$, and $0 \leq k \leq m-1$. When $k=m-1$, $m+n-k-1=p+2m-2-(m-1)=p+m-1$, $m+n-2k-1=p$, $m-k-1=0$, and $n-k-1=p+m-1-m=p-1$.  Thus
\[\nu_p\left(\binom{m+n-k-1}{k}\right)=\nu_p\left(\binom{p+m-1}{m-1}\right)=0
=\nu_p\left(\binom{p-1}{0}\right)=\nu_p\left(\binom{m+n-2k-2}{m-k-1}\right).\]

Assume $0 \leq k \leq m-2$.  Then
\[p+1 \leq m+n-2 k-2 < m+n-2 k-1 \leq m+n-k-1 \leq p+2m-2 \leq 2p-1\]
and 
\[p \leq n-k-1 \leq p+m-2<2 p.\]
Hence $\nu_p((m+n-2k-1)!)=\nu_p((m+n-2k-2)!)=1=\nu_p((m+n-k-1)!)=\nu_p((n-k-1)!)$ in this case.

Clearly $\nu_p(k!)=0=\nu_p((m-k-1)!)$.  It follows that
\[\nu_p \left(\binom{m+n-k-1}{k}\right)=0=\nu_p \left(\binom{m+n-2 k-2}{m-k-1} \right)\]
when $0 \leq k \leq m-2$.

In the remaining part of the proof we use the fact that $\nu_p(\binom{a}{b})$ is the number of carries in adding $b$ and $a-b$ in base $p$.

Suppose that $1 < m \leq n' \leq p+1-m$ and that $n=n'+r p$ where $r$ is a positive integer. We must show that then number of carries in adding $m+n -2k-1=m+n'-2k-1+rp$ and $k$ equals the number of carries in adding $m-k-1$ and $n-k-1=n'-k-1+r p$.  Since, by above, there are none in adding $m+n'-2k-1$ and $k$ and none in adding $m-k-1$ and $n'-k-1$ the result follows.

Suppose that $1<m \leq (p+1)/2$, $n'=p+m-1$, and $n=n'+r p$ where $r$ is a positive integer.  We must show that then number of carries in adding $m+n -2k-1=m+n'-2k-1+rp$ and $k$ equals the number of carries in adding $m-k-1$ and $n-k-1=n'-k-1+r p$.  Since, by above, there is no carry in adding $m+n'-2k-1$ and $k$ and none in adding $m-k-1$ and $n'-k-1$, the result is immediate.
\end{proof}

\section{The case of $m=n=p^t+\frac{p^t+1}{2}$}\label{inductstep}

\begin{theorem}\label{Tpt}
Proposition~\ref{Prop} holds for $m=n=p^t+\frac{p^t+1}{2}$.
\end{theorem}

We will need two lemmas.  First we consider a subcase of $\frac{p^t+1}{2} \leq k \leq p^t-1$.

\begin{lemma}\label{Lpt1}
When an integer $k$ satisfies $\frac{p^t+1}{2} \leq k \leq \frac{p^t+1}{2}+p^{t-1}-1$,
\[\nu_p \left( \binom{2m-k-1}{k} \right)=\nu_p\left( \binom{2m -(k-p^{t-1})-1}{k-p^{t-1}} \right)+1\]
and
\[\nu_p \left( \binom{2m-2k-2}{m-k-1} \right)=\nu_p\left( \binom{2m -2(k-p^{t-1})-2}{m-(k-p^{t-1})-1} \right)+1.\]
\end{lemma}
\begin{proof}
Since
\[\frac{p^t+1}{2}-1=\frac{p^t-1}{2}=\frac{p-1}{2}p^{t-1}+\frac{p-1}{2}p^{t-2}+\dots +\frac{p-1}{2}p+\frac{p-1}{2},\]
\[\frac{p^t+1}{2}+p^{t-1}-1=\frac{p+1}{2}p^{t-1}+\frac{p-1}{2}p^{t-2}+\dots +\frac{p-1}{2}p+\frac{p-1}{2}.\]
Thus $k=b_{t-1} p^{t-1}+\dots + b_1 p+b_0$, where $b_{t-1}=\frac{p-1}{2}$ or $b_{t-1}=\frac{p+1}{2}$.
Now
\[-\frac{p^t+1}{2}-p^{t-1}+1 \leq -k \leq -\frac{p^t+1}{2},\]
implies
\[2 p^t+\frac{p^t-1}{2}-p^{t-1}+1 \leq 2m-1-k \leq 2 p^t+\frac{p^t-1}{2}\]
and
\[p^t+(p-2)p^{t-1}+1 \leq 2m -1-2k \leq p^t+p^t-1.\]

Therefore $2m-2k-1=p^t+a_{t-1} p^{t-1}+\dots+a_1p+a_0$, where $a_{t-1}=p-2$ or $a_{t-2}=p-1$.  Note if $k > \frac{p+1}{2}p^{t-1}$, so $b_{t-1}=\frac{p+1}{2}$, then
\[2 m-2 k-1=3 p^t-2k < 3 p^t-(p^t+p^{t-1})=p^t+(p-1)p^t,\]
and $a_{t-1}=p-2$.  If $k=\frac{p+1}{2} p^{t-1}$, then $b_{t-1}=\frac{p+1}{2}$ and $a_{t-1}=p-1$.  On the other hand, if $k < \frac{p+1}{2}p^{t-1}$, $b_{t-1}=\frac{p-1}{2}$ and
\[2m-2k-1=3 p^t-2k>3 p^t-p^t-p^{t-1}=p^t+(p-1)p^{t-1},\]
so $a_{t-1}=p-1$.  For all such $k$, $a_{t-1}+b_{t-1}\geq p+\frac{p-3}{2} \geq p$.  Hence there is a carry from the $p^{t-1}$ digit when we add $2m-2k-1$ and $k$.

Now $2 m-2(k-p^{t-1})-1=2 p^t+a'_{t-1}p^{t-1}+a_{t-2} p^{t-2}+\dots+a_1p+a_0$ where $a'_{t-1}=0$ when $a_{t-1}=p-2$ and $a'_{t-1}=1$ when $a_{t-1}=p-1$.  And $k-p^{t-1}=b'_{t-1}p^{t-1}+b_{t-2}p^{t-2}+\dots+b_1 p+b_0$ where $b'_{t-1}=b_{t-1}-1$.  Hence $a'_{t-1}+b'_{t-1}\leq \frac{p+1}{2}$ with strict inequality unless $k=\frac{p+1}{2} p^{t-1}$.  Since $\frac{p+1}{2}<p-1$ if $p >3$, there is no carry from the $p^{t-1}$ when we add $2m-2(k-p^{t-1})-1$ and $k-p^{t-1}$ in this case.  When $p=3$ and $k=\frac{p+1}{2} p^{t-1}=2 \cdot 3^{t-1}$, there is a carry from the $3^{t-1}$ digit in adding $2m-2k-1=3^t+2 \cdot 3^t$ and $k=2 \cdot 3^{t-1}$ and no carry in adding $2m-2(k-p^{t-1})-1=2 \cdot 3^t+3^{t-1}$ and $k-p^{t-1}=3^{t-1}$.

From what we have done, it follows that there is exactly one more carry in adding $2m-2k-1$ and $k$ than in adding $2m-2(k-p^{t-1})-1$ and $k-p^{t-1}$.  This proves
\[\nu_p \left( \binom{2m-k-1}{k} \right)=\nu_p\left( \binom{2m -(k-p^{t-1})-1}{k-p^{t-1}} \right)+1\]
when $\frac{p^t+1}{2} \leq k \leq \frac{p^t+1}{2}+p^{t-1}-1$.

Now $p^t-p^{t-1} \leq m-k-1 \leq p^t-1$.  Therefore
\[m-k-1=(p-1)p^{t-1}+c_{t-2}p^{t-2}+\dots +c_1 p+c_0,\]
and
\[m-(k-p^{t-1})-1=p^t+0 \cdot p^{t-1}+c_{t-2}p^{t-2}+\dots +c_1 p+c_0.\]
Hence there is one more carry in adding $m-k-1$ to itself than in adding $m-(k-p^{t-1})-1$ to itself.  This proves
\[\nu_p \left( \binom{2m-2k-2}{m-k-1} \right)=\nu_p\left( \binom{2m -2(k-p^{t-1})-2}{m-(k-p^{t-1})-1} \right)+1\]
when $\frac{p^t+1}{2} \leq k \leq \frac{p^t+1}{2}+p^{t-1}-1$.
\end{proof}

Next we handle the rest of the case $\frac{p^t+1}{2} \leq k \leq p^t-1$.
\begin{lemma}\label{Lpt2}
When $\frac{p^t+1}{2} \leq k \leq \frac{p^t+1}{2}+p^{t-1}-1$ and $k+j p^{t-1} \leq p^t-1$,
\[\nu_p \left( \binom{2m-k-1}{k} \right)=\nu_p \left( \binom{2m-(k+j p^{t-1})-1}{k+j p^{t-1}} \right)\]
and
\[\nu_p \left( \binom{2m-2k-2}{m-k-1} \right)=\nu_p \left( \binom{2m-2(k+j p^{t-1})-2}{m-(k+j p^{t-1})-1} \right).\]
\end{lemma}
\begin{proof}
Using the notation of the previous lemma, $m-k-1=(p-1)p^{t-1}+c_{t-2}p^{t-2}+\dots +c_1 p+c_0$.  Therefore $m-(k+j p^{t-1})-1=(p-1-j)p^{t-1}+c_{t-2}p^{t-2}+\dots +c_1 p+c_0$.  Now
\[j p^{t-1} \leq p^t-1-k \leq p^t-1-\frac{p^t+1}{2}=\frac{p^t-3}{2}
=\frac{p-1}{2} p^{t-1}+\frac{p-1}{2} p^{t-2}+\dots +\frac{p-1}{2}p+\frac{p-3}{2}.\]
Hence $j \leq \frac{p-1}{2}$ and $p-1-j \geq \frac{p-1}{2}$. If $j<\frac{p-1}{2}$, then the number of carries in adding $m-k-1$ to itself equals the number of carries in adding $m-(k+j p^{t-1})-1$ to itself. 

If $j=\frac{p-1}{2}$, then
\[\frac{p^t+1}{2} \leq  k\leq p^t-1-\frac{p-1}{2}p^{t-1}=\frac{p^t+p^{t-1}}{2}-1\]
and
\[p^t-\frac{p^{t-1}-1}{2} =(p-1) p^{t-1}+\frac{p^{t-1}+1}{2} \leq m-k-1 \leq p^t-1.\]
Thus in adding $m-k-1$ to itself there is a carry out of the $p^{t-2}$ digit.  Because $p-1-j=\frac{p-1}{2}$, when adding $m-(k+\frac{p-1}{2}p^{t-1})-1$ to itself, that carry out of the $p^{t-2}$ digit causes a carry out of the $p^{t-1}$ digit.  We have shown that the number of carries in adding $m-k-1$ to itself equals the number of carries in adding $m-(k+\frac{p-1}{2}p^{t-1})-1$ to itself.

From the previous lemma,
$2m-2k-1=p^t+a_{t-1} p^{t-1}+\dots+a_1p+a_0$, where $a_{t-1}=p-2$ or $a_{t-2}=p-1$, and
$k=b_{t-1} p^{t-1}+\dots + b_1 p+b_0$, where $b_{t-1}=\frac{p-1}{2}$ or $b_{t-1}=\frac{p+1}{2}$ and $a_{t-1}+b_{t-1} \geq p+\frac{p-3}{2} \geq p$.  Therefore
\[2m-2(k+j p^{t-1})-1=p^t+(a_{t-1}-2j) p^{t-1}+a_{t-2} p^{t-2}+\dots+ a_1 p+a_0\]
and $k+j p^{t-1}=(b_{t-1}+j) p^{t-1}+b_{t-2} p^{t-2}+\dots+b_1 p+b_0$.  If $j<\frac{p-1}{2}$, then the number of carries in adding $2m-2k-1$ and $k$ equals the number of carries in adding $2m-2(k+j p^{t-1})-1$ and $k+j p^{t-1}$.  

Assume that $j=\frac{p-1}{2}$.  Then $\frac{p^t+1}{2} \leq  k\leq p^t-1-\frac{p-1}{2}p^{t-1}=\frac{p^t+p^{t-1}}{2}-1$ from above, clearly $b_{t-1}=\frac{p-1}{2}$,   and $a_{t-1}=p-1$ because $p^t+(p-1)p^{t-1}+2 \leq 2m-2k-1 \leq 2 p^t-1$. 

Then the $p^{t-1}$ digit of $k+\frac{p-1}{2}p^{t-1}$ is $p-1$ and the $p^{t-1}$ digit of $2m-2(k+\frac{p-1}{2}p^{t-1})-1$ is $0$.  We must show there is a carry out of the $p^{t-2}$ digit in adding $2m-2(k+\frac{p-1}{2}p^{t-1})-1$ and $k+\frac{p-1}{2}p^{t-1}$ which results in a carry out of the $p^{t-1}$ digit.  Of course, there must be a carry out of the $p^{t-2}$ digit in adding $2m-2k-1$ and $k$.  It comes down to showing that when $\frac{p^t+1}{2} \leq  k\leq p^t-1-\frac{p-1}{2}p^{t-1}=\frac{p^t+p^{t-1}}{2}-1$,
\[(a_{t-2}p^{t-2}+\dots+a_1 p+a_0)+(b_{t-2} t^{t-2}+\dots+b_1+b_0) \geq p^{t-1}.\]  
Now
\begin{align*}
(a_{t-2}p^{t-2}+\dots+a_1 p+a_0)&+(b_{t-2} t^{t-2}+\dots+b_1+b_0) \geq p^{t-1}\\
&=(2m-2k-1-p^t-(p-1)p^{t-1})+(k-\frac{p-1}{2} p^{t-1})\\
&=3 p^t-k-p^t-(p-1) p^{t-1}-\frac{p-1}{2} p^{t-1}\\
&=\frac{p^t+3 p^{t-1}}{2}-k\\
&\geq \frac{p^t+3 p^{t-1}}{2}-(\frac{p^t+p^{t-1}}{2}-1)\\
&=p^{t-1}+1.
\end{align*}
\end{proof}

\begin{proof}[Proof of Theorem~\ref{Tpt}]
We break this proof into three cases, the second of which has subcases:
\begin{enumerate}
\item $0 \leq k \leq \frac{p^t+1}{2}-1$
\item $\frac{p^t+1}{2} \leq k \leq p^t-1$
\item $p^t \leq k \leq m-1$
\end{enumerate}

Assume that $0 \leq k \leq \frac{p^t-1}{2}$.  When $k=0$, $\binom{2m-1}{0}=1$, so $\nu_p(\binom{2m-1}{0})=0$.  On the other hand $\binom{2m-2}{m-1}=\binom{ 3 p^t-1}{p^t+\frac{p^t-1}{2}}$, and since there are no carries in adding $p^t+\frac{p^t-1}{2}$ to itself, $\nu_p(\binom{2m-2}{m-1})=0$.  Thus the result of Proposition~\ref{Prop} holds when $k=0$.

Assume then that $1 \leq k \leq \frac{p^t-1}{2}$. Then $2m-2k-1=(3 p^t+1)-2k-1=2 p^t+p^t-2k$,
\[2 p^t+1 \leq 2m-2k-1 \leq 2 p^t+(p^t-2),\]
and
\[2 p^t+\frac{p^t+1}{2} \leq 2m-k-1 \leq 3 p^t-1.\]

Also, $m-k-1=p^t+\frac{p^t+1}{2}-k-1=p^t+\frac{p^t-1}{2}-k$ and
\[p^t \leq m-k-1 \leq p^t+\frac{p^t-3}{2}.\]

Now the number of carries in adding $2 m-2k-1$ and $k$ equals the number of carries in adding $p^t-2k$ and $k$, and the number of carries in adding $m-k-1$ to itself equals the number of carries in adding $\frac{p^t-1}{2}-k$ to itself.

If we set $m'=\frac{p^t+1}{2}$, then $2m'-2k-1=p^t-2k$ and $m'-k-1=\frac{p^t-1}{2}-k$.  Then $(m',m')\in S_{t-1}$ and so the result (*) is true for all integers $k \in [1,m'-1]$ by induction.

Now assume that $\frac{p^t+1}{2} \leq k \leq p^t-1$.  When $k \leq \frac{p^t+1}{2}+p^{t-1}-1$, $0 \leq k-p^{t-1} \leq \frac{p^t+1}{2}-1$. Thus
\[\nu_p\left( \binom{2m -(k-p^{t-1})-1}{k-p^{t-1}} \right)=
\nu_p\left( \binom{2m -2(k-p^{t-1})-2}{m-(k-p^{t-1})-1} \right)\]
by our work above.  Hence, by Lemma~\ref{Lpt1},
\[\nu_p \left( \binom{2m-k-1}{k} \right)=\nu_p \left( \binom{2m-2k-2}{m-k-1} \right)\]
for every integer $k$ satisfying $\frac{p^t+1}{2} \leq k \leq  \frac{p^t+1}{2}+p^{t-1}-1$.

When $\frac{p^t+1}{2}+p^{t-1} \leq k \leq  p^t-1$, then $k=k'+j p^{t-1}$ where $\frac{p^t+1}{2} \leq k' \leq \frac{p^t+1}{2}+p^{t-1}-1$ and $j$ is a positive integer.  By the previous paragraph,
\[\nu_p \left( \binom{2m-k'-1}{k'} \right)=\nu_p \left( \binom{2m-2k'-2}{m-k'-1} \right).\]
Then, by Lemma~\ref{Lpt2},
\[\nu_p \left( \binom{2m-k-1}{k} \right)=\nu_p \left( \binom{2m-2k-2}{m-k-1} \right)\]
for every integer $k$ satisfying $\frac{p^t+1}{2}+p^{t-1} \leq k \leq  p^t-1$.

Assume now that $p^t \leq k \leq p^t+\frac{p^t+1}{2}-1$.  Then $1 \leq 2m-2k-1 \leq p^t$.  When $k=p^t$, $2 m-2k-1=p^t$ and there is no carry in adding $2m-2k-1$ and $k$ and no carry in adding $2m-2(k-p^t)-1$ and $k-p^t=0$.  Assume that $k>p^t$.  Then $2 m-2 k-1<p^t$ implying $2m-2k-1=a_{t-1} p^{t-1}+\dots+a_1p+a_0$, and $k=p^t+b_{t-1}p^{t-1}+\dots+b_1 p+b_0$.  But $2 m-2k -1=3 p^t-2k=p^t-2(k-p^t)$.  Thus $p^t-2(k-p^t)=a_{t-1} p^{t-1}+\dots+a_1p+a_0$.  Since $k-p^t=b_{t-1}p^{t-1}+\dots+b_1 p+b_0$, the number of carries in adding $2m-2k-1$ and $k$ equals the number of carries in adding $p^t-2(k-p^t)$ and $k-p^t$.  Thus
\[\nu_p \left( \binom{2m-k-1}{k} \right)=
\nu_p\left( \binom{(p^t+1)-2(k-p^t)-1}{k-p^t}\right)\]
when $p^t \leq k \leq p^t+\frac{p^t+1}{2}-1$.

On the other hand $m-k-1=\frac{p^t+1}{2}-(k-p^t)-1$.  Thus
\[\nu_p \left( \binom{2m-2k-2}{m-k-1} \right)=
\nu_p \left( \binom{(p^t+1)-2(k-p^t)-2}{\frac{p^t+1}{2}-(k-p^t)-1} \right)\]
when $p^t \leq k \leq p^t+\frac{p^t+1}{2}-1$.  If we set $m'=\frac{p^t+1}{2}$, then $(m',m') \in S_{t-1}$ and hence
\[\nu_p\left( \binom{(p^t+1)-2(k-p^t)-1}{k-p^t}\right)=
\nu_p \left( \binom{(p^t+1)-2(k-p^t)-2}{\frac{p^t+1}{2}-(k-p^t)-1} \right).\]
It follows that
\[\nu_p \left( \binom{2m-k-1}{k} \right)=
\nu_p \left( \binom{2m-2k-2}{m-k-1} \right)\]
when $p^t \leq k \leq p^t+\frac{p^t+1}{2}-1$.

\end{proof}

\section{The case of $m=n=i p^t+\frac{p^t+1}{2}$}

The next lemma reduces the case $(m,n)=(i p^t+\frac{p^t+1}{2},i p^t+\frac{p^t+1}{2})$ to the case $(m,n)=(\frac{p^t+1}{2},\frac{p^t+1}{2})$ or to the case $(m,n)=(p^t+\frac{p^t+1}{2},p^t+\frac{p^t+1}{2})$.

\begin{lemma}
Here $m=ip^t+(p^t+1)/2$ with $1 \leq i \leq \frac{p-1}{2}$.
\begin{enumerate}
\item When $0 \leq j \leq i$ and $j p^t \leq k \leq j p^t+\frac{p^t+1}{2}-1$,
\[\nu_p \left( \binom{2m-2k-2}{m-k-1} \right)=
\nu_p \left( \binom{(p^t+1)-2(k-j p^t)-2}{\frac{p^t+1}{2}-(k-j p^t)-1}\right)\]
and
\[\nu_p \left( \binom{2m-k-1}{k} \right)=
\nu_p \left( \binom{(p^t+1)-(k-j p^t)-1}{k-j p^t} \right).\]

\item When $0 \leq j \leq i$ and $j p^t+\frac{p^t+1}{2} \leq k \leq (j+1) p^t-1$,
\[\nu_p \left( \binom{2m-2k-2}{m-k-1} \right)=
\nu_p \left(\binom{(3p^t+1)-2(k-j p^t)-2}{p^t+\frac{p^t+1}{2}-(k-j p^t)-1}\right)\]
and
\[\nu_p \left( \binom{2m-k-1}{k} \right)=
\nu_p \left( \binom{(3p^t+1)-(k-j p^t)-1}{k-j p^t}\right).\]
\end{enumerate}
\end{lemma}
\begin{proof}
(1) First suppose that $j p^t \leq k \leq j p^t+\frac{p^t+1}{2}-1$.  Then
\[(i-j) p^t \leq m-k-1 \leq (i-j) p^t+\frac{p^t-1}{2},\]
or
\[(i-j) p^t \leq i p^t+\frac{p^t-1}{2}-k \leq (i-j) p^t+\frac{p^t-1}{2}.\]
Hence
\[0 \leq \frac{p^t+1}{2}-1-(k-j p^t) \leq \frac{p^t+1}{2}-1.\]
Therefore if $m-k-1=(i-j) p^t+c_{t-1} p^{t-1}+\dots + c_1 p +c_0$, then
\[i p^t +\frac{p^t+1}{2}-1-k=(i-j) p^t+c_{t-1} p^{t-1}+\dots + c_1 p +c_0,\]
and
\[(i-j) p^t+\frac{p^t+1}{2}-1-(k-j p^t)=(i-j) p^t+c_{t-1} p^{t-1}+\dots + c_1 p +c_0,\]
from which it follows that
\[\frac{p^t+1}{2}-1-(k-j p^t)=c_{t-1} p^{t-1}+\dots + c_1 p +c_0.\]
Hence the number of carries in adding $m-k-1$ to itself equals the number of carries in adding $\frac{p^t+1}{2}-1-(k-j p^t$ to itself.  Thus

\[\nu_p \left( \binom{2m-2k-2}{m-k-1} \right)=
\nu_p \left(\binom{(3p^t+1)-2(k-j p^t)-2}{p^t+\frac{p^t+1}{2}-(k-j p^t)-1}\right).\]

Also
\[2(i-j) p^t+1 \leq 2m-2k-1 \leq (2(i-j)+1) p^t,\]
so
\[2 i p^t+1 \leq 2m-2(k-j p^t)-1 \leq (2 i+1)p^t,\]
and
\[0 \leq k-j p^t \leq \frac{p^t+1}{2}-1.\]
If $k=j p^t$, then $2 m-2k-1=(2(i-j)+1) p^t$, and there are zero carries in adding $2 m-2 k-1$ and $k$.  Also $2 m-2(k-j p^t)-1=(2i+1) p^t$ and $k-j p^t=0$, so there are zero carries in adding $2m-2(k-j p^t)-1$ and $k-j p^t$.  We can assume that $k>j p^t$, so $2 m-2k-1<(2(i-j)+1)p^t$.  Then $2 m-2k-1=2(i-j) p^t+a_{t-1}p^{t-1}+\dots a_1 p+a_0$ and $k=j p^t+b_{t-1} p^{t-1}+\dots + b_1 p+b_0$.  Since $2m-2k-1=(2i+1) p^t-2k=2 (i-j) p^t+p^t-2(k-j p^t)$.  Thus $p^t-2(k-j p^t)=a_{t-1}p^{t-1}+\dots a_1 p+a_0$ and $k-j p^t=b_{t-1} p^{t-1}+\dots + b_1 p+b_0$.  This shows that the number of carries in adding $2m-2k-1$ to $k$ equals the number of carries in adding $p^t-2(k-j p^t)$ to $k-j p^t$.  Hence
\[\nu_p \left( \binom{2m-k-1}{k} \right)=
\nu_p \left( \binom{(p^t+1)-(k-j p^t)-1}{k-j p^t} \right).\]

(2) Now suppose that $j p^t+\frac{p^t+1}{2} \leq k \leq (j+1) p^t-1$.  Note $j<i$ here.  Then
\[(2(i-j)-1)p^t+2 \leq 2m-2k-1 \leq 2(i-j)p^t-1\]
and $2 m-2 k-1=(2(i-j)-1) p^t+a_{t-1} p^{t-1}+ \dots + a_1 p +a_0$.  But
\[2m-2k-1=(2i+1)p^t-2k=2 i p^t+p^t-2k=(2(i-j)-1)p^t+2 p^t-2(k-j p^t).\]  Therefore $3 p^t-2(k-j p^t)=p^t+
a_{t-1} p^{t-1}+ \dots + a_1 p +a_0$ and
\[2 m-2(k- jp^t)-1=(2i-1) p^t+a_{t-1} p^{t-1}+ \dots + a_1 p +a_0.\]
Now $k =j p^t+b_{t-1} p^{t-1}+\dots+b_1 p+b_0$ and $k-j p^t=b_{t-1} p^{t-1}+\dots+b_1 p+b_0$.  Hence the number of carries in adding $2 m-2k-1$ to $k$ equals the number of carries in adding $3 p^t-2(k-j p^t)$ to $k-j p^t$.  Thus
\[\nu_p \left( \binom{2m-k-1}{k} \right)=
\nu_p \left( \binom{(3p^t+1)-(k-j p^t)-1}{k-j p^t}\right).\]

Now
\[(i-j-1)p^t +\frac{p^t+1}{2} \leq m-k-1 \leq (i-j-1)p^t +p^t-1\]
and
\[m-k-1=i p^t+\frac{p^t+1}{2}-k-1=(i-j-1)p^t+p^t+\frac{p^t+1}{2}-(k-j p^t)-1.\]
If $m-k-1=(i-j-1) p^t+c_{t-1} p^{t-1}+\dots +c_1p +c_0$, then
\[p^t+\frac{p^t+1}{2}-(k-j p^t)-1=c_{t-1} p^{t-1}+\dots +c_1p +c_0.\] This shows that the number of carries in adding $m-k-1$ to itself equals the number of carries in adding $p^t+\frac{p^t+1}{2}-(k-j p^t)-1$ to itself.
Hence
\[\nu_p \left( \binom{2m-2k-2}{m-k-1} \right)=
\nu_p \left(\binom{(3p^t+1)-2(k-j p^t)-2}{p^t+\frac{p^t+1}{2}-(k-j p^t)-1}\right).\]

\end{proof}

\begin{corollary}\label{i+i+}
Proposition~\ref{Prop} holds for $(m,n)=(i p^t+\frac{p^t+1}{2},i p^t+\frac{p^t+1}{2})$.
\end{corollary}

\begin{proof}
When $j p^t \leq k \leq j p^t+\frac{p^t+1}{2}-1$,
\[\nu_p \left( \binom{2m-2k-2}{m-k-1} \right)=
\nu_p \left( \binom{(p^t+1)-2(k-j p^t)-2}{\frac{p^t+1}{2}-(k-j p^t)-1}\right)\]
and 
\[\nu_p \left( \binom{2m-k-1}{k} \right)=
\nu_p \left( \binom{(p^t+1)-(k-j p^t)-1}{k-j p^t} \right).\]
Set $m' =\frac{p^t+1}{2}$.  Then $(m',m') \in S_{t-1}$, so Proposition~\ref{Prop} holds for $(m',m')$ by inductive assumption and
\[\nu_p \left( \binom{(p^t+1)-2(k-j p^t)-2}{\frac{p^t+1}{2}-(k-j p^t)-1}\right)=
\nu_p \left( \binom{(p^t+1)-(k-j p^t)-1}{k-j p^t} \right).\]
Thus 
\[\nu_p \left( \binom{2m-2k-2}{m-k-1} \right)=
\nu_p \left( \binom{2m-k-1}{k} \right)\] for such $k$.

When $j p^t+\frac{p^t+1}{2} \leq k \leq (j+1) p^t-1$,
\[
\nu_p \left( \binom{2m-2k-2}{m-k-1} \right)=\nu_p \left(\binom{(3p^t+1)-2(k-j p^t)-2}{p^t+\frac{p^t+1}{2}-(k-j p^t)-1}\right)\]
and
\[\nu_p \left( \binom{2m-k-1}{k} \right)=
\nu_p \left( \binom{(3p^t+1)-(k-j p^t)-1}{k-j p^t}\right).
\]
Set $m'=p^t+\frac{p^t+1}{2}$.  Then Proposition~\ref{Prop} holds for $(m',m')$ by Theorem~\ref{Tpt}.  Hence
\[\nu_p \left( \binom{(3 p^t+1)-2(k-j p^t)-2}{ p^t+\frac{p^t+1}{2}-(k-j p^t)-1}\right)=\nu_p \left( \binom{(3p^t+1)-(k-j p^t)-1}{k-j p^t}\right).
\]
Thus 
\[\nu_p \left( \binom{2m-2k-2}{m-k-1} \right)=
\nu_p \left( \binom{2m-k-1}{k} \right)
\] for all such $k$.
\end{proof}

\section{The case of $(m,n)=(i p^t+\frac{p^t+1}{2},j p^t+\frac{p^t+1}{2})$}
 Here $1 \leq i \leq j \leq p-1-i$, so $i \leq \frac{p-1}{2}$.
\begin{lemma}\label{lemi+j+}
For every integer $k \in [0,m-1]$,
\[\nu_p\left( \binom{m+n-k-1}{k} \right)=
\nu_p \left(\binom{2m-k-1}{k} \right)\]
and
\[\nu_p\left( \binom{m+n-2k-2}{m-k-1}\right)=
\nu_p\left( \binom{2m-2k-2}{m-k-1}\right).\]
\end{lemma}

\begin{proof}
If $2m-2k-1=a_t p^t+a_{t-1} p^{t-1}+\dots a_1 p+a_0$ and $k=b_t p^t+b_{t-1} p^{t-1}+\dots + b_1 p+b_0$, then $m+n-2k-1=(a_t+j-i)p^t+a_{t-1} p^{t-1}+\dots a_1 p+a_0$.  If $k=0$, then the number of carries in adding $m+n-1$ and $0$ equals the number of carries in adding $2 m-1$ and $0$.  Assume that $k>0$, in which case  $a_t \leq 2 i \leq i+j \leq p-1$ and $m+n-k-1<p^{t+1}$.  Thus the number of carries in adding $m+n-2k-1$ and $k$ equals the number of carries in adding $2m-2k-1$ and $k$, and therefore
\[\nu_p\left( \binom{m+n-k-1}{k} \right)=
\nu_p \left(\binom{2m-k-1}{k} \right).\]
If $m-k-1=c_t p^t+c_{t-1} p^{t-1}+\dots + c_1 p+c_0$, then $n-k-1=(c_t+j-i) p^t+c_{t-1} p^{t-1}+\dots + c_1 p+c_0$ and $m+n-2k-2 \leq m+n-2 < p^{t+1}$.  Thus the number of carries in adding $m-k-1$ and $n-k-1$ equals the number of carries in adding $m-k-1$ to itself, and therefore
\[\nu_p\left( \binom{m+n-2k-2}{m-k-1}\right)=
\nu_p\left( \binom{2m-2k-2}{m-k-1}\right).\]
\end{proof}

\begin{corollary}\label{Cori+j+}
Proposition~\ref{Prop} holds for $(m,n)=(i p^t+\frac{p^t+1}{2},j p^t+\frac{p^t+1}{2})$.
\end{corollary}
\begin{proof}
Since, by Corollary~\ref{i+i+},
\[\nu_p \left(\binom{2m-k-1}{k} \right)=\nu_p\left( \binom{2m-2k-2}{m-k-1}\right)\]
for every integer $k \in [0,m-1]$, the result follows from Lemma~\ref{lemi+j+}.
\end{proof}

\section{Interlude}

In the next three sections, $m=ip^t+\frac{p^t \pm 1}{2}$, $n =j p^t+\frac{p^t \pm 1}{2}$, $m \leq n$, and $(m,n) \in S'_t$.  We will relate $\nu_p(\binom{m+n-k-1}{k})$ to $\nu_p(\binom{m_1+n_1-k-1}{k})$ and $\nu_p(\binom{m+n-2k-2}{m-k-1})$ to $\nu_p(\binom{m_1+n_1-2 k-2}{m_1-k-1})$ where $(m_1,n_1)=(ip^t+\frac{p^t + 1}{2},jp^t+\frac{p^t + 1}{2})$.  Since by Corollary~\ref{Cori+j+}, the result holds for $(m_1,n_1)$, it will follow immediately for $(m,n)$.

\section{The case of $(m,n)=(i p^t+\frac{p^t-1}{2},j p^t+\frac{p^t+1}{2})$}
Here $1 \leq i \leq j \leq p-1-i$.

\begin{lemma}\label{lemi-j+}
Let $m_1=m+1$.  For every integer $k \in [0,m-1]$,
\[\nu_p \left( \binom{m+n-k-1}{k}\right)=
\nu_p \left( \binom{m_1+n-k-1}{k} \right)
 \]
and
\[\nu_p \left( \binom{m+n-2k-2}{m-k-1} \right)=
\nu_p \left( \binom{m_1+n-2k-2}{m_1-k-1} \right)
.\]
\end{lemma}

\begin{proof}
Here $m_1+n=(i+j+1)p^t+1$.  Then $m+n-2k-1=m_1+n-2(k+1)=(i+j+1)p^t-2k-1$.  We must show that the number of carries in adding $m_1+n-2(k+1)$ and $k$ equals the number of carries in adding $m_1+n-2k-1$ and $k$.  It suffices to show that
\[\nu_p \left( \binom{(i+j+1)p^t-k-1}{k}\right)=
\nu_p \left( \binom{(i+j+1)p^t-k}{k}\right)\]
for every integer $k \in [0,m-1]$.  Note that
\[\binom{(i+j+1)p^t-k}{k}=\frac{(i+j+1)p^t-k}{(i+j+1)p^t-2k}
\binom{(i+j+1)p^t-k-1}{k}.\]
Since $p$ is an odd prime, the exact power of $p$ that divides $(i+j+1)p^t-k$ equals the exact power of $p$ that divides $(i+j+1)p^t-2k$ and the result follows.

Also, $m-k-1=m_1-k-2$ and $m+n-2k-2=m_1+n-2k-3=(i+j+1)p^t-2k-2$.  We must show
\[\nu_p \left( \binom{(i+j+1)p^t-2k-2}{m_1-k-2} \right)=
\nu_p \left( \binom{(i+j+1)p^t-2k-1}{m_1-k-1}\right). \]
for every integer $k \in [0,m_1-2]$.  Note that
\[ \binom{(i+j+1)p^t-2k-1}{m_1-k-1}=
\frac{(i+j+1)p^t-2k-1}{m_1-k-1} \binom{(i+j+1)p^t-2k-2}{m_1-k-2}.\]
Since $(i+j+1)p^t+1 -2k-2=(j-i) p^t +(2i+1)p^t+1-2k-2=(j-i)p^t+2(m_1-k-1)$ and $p$ is an odd prime, the exact power of $p$ dividing $(i+j+1)p^t-2k-1$ equals the exact power of $p$ dividing $m_1-k-1$ and the result follows.
\end{proof}

\section{The case of $(m,n)=(i p^t+\frac{p^t+1}{2},j p^t+\frac{p^t-1}{2})$}
Here $1 \leq i<j \leq p-1-i$.

\begin{lemma}
Let $n_1=n+1$.   For every integer $k \in [0,m-1]$,
\[
\nu_p \left( \binom{m+n-k-1}{k} \right)=
\nu_p \left( \binom{m+n_1-k-1}{k} \right)\]
and
\[
\nu_p \left( \binom{m+n-2k-2}{m-k-1}\right)=
\nu_p \left( \binom{m+n_1-2k-2}{m-k-1} \right).
\]
\end{lemma}
\begin{proof}
The first equality above follows from the first equality in Lemma~\ref{lemi-j+}.
Here $m+n=(i+j)p^t$.  Since
\[\binom{m+n_1-2k-2}{m-k-1}=\frac{m+n_1-2k-2}{n_1-k-1} \binom{m+n-2k-2}{m-k-1}\]
and
\[\frac{m+n_1-2k-2}{n_1-k-1}=\frac{(i-j)p^t+2(n_1-k-1)}{n_1-k-1},\]
$\nu_p((i-j)p^t+2(n_1-k-1))=\nu_p(n_1-k-1)$ and the second equality holds.
\end{proof}

\section{The case of $(m,n)=(i p^t+\frac{p^t-1}{2},j p^t+\frac{p^t-1}{2})$}
Here $1 \leq i \leq j \leq p-1-i$.

\begin{lemma}\label{lemi-j-}
Let $m_1=m+1$, $n_1=n+1$.  For every integer $k \in [0,m-1]$,
\[
\nu_p \left( \binom{m+n-k-1}{k} \right)=
\nu_p \left( \binom{m_1+n_1-(k+1)-1}{k+1} \right)
\]
and
\[
\nu_p \left( \binom{m+n-2k-2}{m-k-1}\right)=
\nu_p \left( \binom{m_1+n_1-2(k+1)-2}{m_1-(k+1)-1} \right)
.
\]
\end{lemma}

\begin{proof}
Now $m+n=m_1+n_1-2=(i+j+1)p^t-1$.  We must show that
\[\nu_p \left( \binom{(i+j+1)p^t-k-2}{k} \right)=
\nu_p \left( \binom{(i+j+1)p^t-k-1)}{k+1} \right)\]
for every integer $k \in [0,m-1]$.  Note that
\[
\binom{(i+j+1)p^t-k-1)}{k+1}=\frac{(i+j+1)p^t-k-1}{k+1}\binom{(i+j+1)p^t-k-2}{k}. 
\]
Since the exact power of $p$ dividing $(i+j+1)p^t-k-1$ equals the exact power of $p$ dividing $k+1$, the result follows.  Since
$m+n-2k-2=m_1+n_1-2(k+1)-2$ and $m-k-1=m_1-(k+1)-1$, the second result is clear.
\end{proof}

\section{The case of $(m,n)=(i p^t+\frac{p^t+1}{2},i p^t+\frac{p^t-1}{2}+p^{t+1})$}

\begin{lemma}\label{Li+i-}
For every integer $k \in [0,m-1]$,
\[
\nu_p \left( \binom{m+n-k-1}{k} \right)=
\nu_p \left( \binom{2m-k-1}{k} \right)\]
and
\[
\nu_p \left( \binom{m+n-2k-2}{m-k-1}\right)=
\nu_p \left( \binom{2m-2k-2}{m-k-1} \right).
\]
\end{lemma}

\begin{proof}
Let $n_1=n+1$.  First we show that for every integer $k \in [0,m-1]$,
\[
\nu_p \left( \binom{m+n-k-1}{k} \right)=
\nu_p \left( \binom{m+n_1-k-1}{k} \right)\]
and
\[
\nu_p \left( \binom{m+n-2k-2}{m-k-1}\right)=
\nu_p \left( \binom{m+n_1-2k-2}{m-k-1} \right).
\]
Here $m+n=2 i p^t+p^{t+1}$.  Now
\[\binom{m+n_1-k-1}{k} =\frac{m+n_1-k-1}{m+n_1-2k-1} \binom{m+n-k-1}{k}\]
and
\[\frac{m+n_1-k-1}{m+n_1-2k-1}=\frac{2 i p^t+p^{t+1}-k}{2i p^t+p^{t+1}-2k}.\]
Clearly $\nu_p(2 i p^t+p^{t+1}-k)=\nu_p(2 i p^t+p^{t+1}-2k)$ and the first equality holds.

Now
\[\binom{m+n_1-2k-2}{m-k-1} =\frac{m+n_1-2k-2}{n_1-k-1} \binom{m+n-2k-2}{m-k-1}\]
and
\[\frac{m+n_1-2k-2}{n_1-k-1} =\frac{-p^{t+1}+2(n_1-k-1)}{n_1-k-1}.\]
Clearly $\nu_p(-p^{t+1}+2(n_1-k-1))=\nu_p(n_1-k-1)$ and the second equality follows.

Finally, since $m+n_1=2m+p^{t+1}$ and there are no carries into the $p^{t+1}$ digit in adding $m+n_1-2k-1$ and $k$ or in adding $n_1-k-1$ and $m-k-1$,
\[\nu_p \left( \binom{m+n_1-k-1}{k} \right)=\nu_p \left( \binom{2m-k-1}{k} \right)\]
and \[\nu_p\left(\binom{m+n_1-2k-2}{m-k-1} \right)=\nu_p \left(\binom{2m-2k-2}{m-k-1}\right),\]
the result follows.

\end{proof}

\begin{corollary}\label{Cori+i-}
Proposition~\ref{Prop} holds for $(m,n)=(i p^t+\frac{p^t+1}{2},i p^t+\frac{p^t-1}{2}+p^{t+1})$.
\end{corollary}
\begin{proof}
Since, by Corollary~\ref{i+i+},
\[\nu_p \left(\binom{2m-k-1}{k} \right)=\nu_p\left( \binom{2m-2k-2}{m-k-1}\right)\]
for every integer $k \in [0,m-1]$, the result follows from Lemma~\ref{Li+i-}.
\end{proof}

\section{The case of $(m,n'+r p^{t+1})$ where $(m,n') \in S'_t$}

Let $m=i p^t+\frac{p^t \pm 1}{2}$ and $n'=j p^t+\frac{p^t \pm 1}{2}$ with $m \leq n'$ ($m-1 \leq n'$?) and $i+j \leq p-1$.

Note that $m+n'-k-1 \leq p^{t+1}$ with equality only occurring when $k=0$, $m=i p^t+\frac{p^t + 1}{2}$,  $n'=j p^t+\frac{p^t + 1}{2}$, and $i+j=p-1$.  Thus in adding $m-k-1$ and $n'-k-1$, there are no carry into the $p^{t+1}$ digit.  Hence the number of carries in adding $m-k-1$ and $n-k-1$ equals the  number of carries in adding $m-k-1$ and $n'-k-1$.  Thus
\[\nu_p \left( \binom{m+n-2k-2}{m-k-1}\right)=\nu_p \left( \binom{m+n'-2k-2}{m-k-1}\right)\]
for every integer $k \in [0,m-1]$.
Clearly, $\binom{m+n-0-1}{0}=1=\binom{m+n'-0-1}{0}$.  Assume that $1 \leq k \leq m-1$.  Then there is no carry into the $p^{t+1}$ digit in adding $m+n'-2k-1$ and $k$.  Hence the number of carries in adding $m+n-2k-1$ and $k$  equals the number of carries in adding $m+n'-2k-1$ and $k$.  Thus
\[\nu_p \left( \binom{m+n-k-1}{k} \right)=
\nu_p \left( \binom{m+n'-k-1}{k} \right)\]
for every integer $k \in [0,m-1]$.
Since
\[\nu_p \left( \binom{m+n'-k-1}{k} \right)=
\nu_p \left( \binom{m+n'-2k-2}{m-k-1}\right),\]
it follows that
\[\nu_p \left( \binom{m+n-k-1}{k} \right)=
\nu_p \left( \binom{m+n-2k-2}{m-k-1}\right)\]
for every integer $k \in [0,m-1]$.

\end{document}